\documentclass[12pt]{amsart}
\usepackage{amsfonts,amsmath,stmaryrd}
\usepackage{amssymb,latexsym}
\usepackage[dvips]{epsfig}
\usepackage{amscd}
\usepackage{amsthm}
\usepackage{pinlabel}
\usepackage{psfrag}

\RequirePackage{color}
\definecolor{myred}{rgb}{0.75,0,0}
\definecolor{mygreen}{rgb}{0,0.5,0}
\definecolor{myblue}{rgb}{0,0,0.65}

\usepackage[hmargin=3cm,vmargin=3.5cm]{geometry}

\usepackage{graphicx}
\usepackage[all]{xy}
\SelectTips{cm}{}

\newcommand{\ot}{\otimes}

\newcommand{\mC}{\mathcal{C}}
\newcommand{\mF}{\mathcal{F}}
\newcommand{\mD}{\mathcal{D}}

\newcommand{\mc}[1]{\mathcal{#1}}

\newcommand{\Bim}{\textbf{Bim}}

\newcommand{\mf}[1]{\mathfrak{#1}}

\newcommand{\Hom}{{\rm Hom}}

\newcommand{\End}{{\rm End}}

\newcommand{\Res}{{\rm Res}}
\newcommand{\Ind}{{\rm Ind}}
\renewcommand{\to}{\rightarrow}
\newcommand{\into}{\hookrightarrow}
\newcommand{\co}{\colon}
\newcommand{\pa}{\partial}
\newcommand{\ig}[2]{\vcenter{\xy (0,0)*{\includegraphics[scale=#1]{figsq/#2}} \endxy}}

\newcommand{\igc}[2]{\begin{center} \includegraphics[scale=#1]{figsq/#2} \end{center}}

\newtheorem*{prop*}{Proposition}
\newtheorem{prop}{Proposition} [section]

\newtheorem{claim}[prop]{Claim}

\theoremstyle{definition}
\newtheorem{defn}[prop]{Definition}
\newtheorem{example}[prop]{Example}

\theoremstyle{remark}
\newtheorem{remark}[prop]{Remark}

\numberwithin{equation}{section}

\let\phi=\varphi
\let\epsilon=\varepsilon

\usepackage{bbm}
\def\C{{\mathbbm C}}

\def\R{{\mathbbm R}}

\def\1{\mathbbm{1}}

\psfrag{I}{$I$}
\psfrag{Icgam}{ $I \setminus \gamma$}
\psfrag{Icgp}{$I \setminus \gamma^\prime$}
\psfrag{Icgcgp}{$I \setminus \{\gamma, \gamma^\prime\}$}
\psfrag{iotaf}{$\iota(f)$}
\psfrag{muBA}{$\mu^B_A$}
\psfrag{paf}{$\pa(f)$}
\psfrag{f}{\small $f$}
\psfrag{A}{$A$}
\psfrag{fg}{$fg$}
\psfrag{B}{$B$}
\psfrag{g}{$g$}
\psfrag{=}{$=$}
\psfrag{B'}{$B'$}
\psfrag{C}{$C$}
\psfrag{deltaBA1}{$\Delta^B_{A\ (1)}$}
\psfrag{deltaBA2}{$\Delta^B_{A\ (2)}$}
\psfrag{xalpha}{$x_\alpha$}
\psfrag{ybeta}{$y_\beta$}
\psfrag{R}{$R$}
\psfrag{R_b}{$R^b$}
\psfrag{R_r}{$R^r$}
\psfrag{R_rb}{$R^{rb}$}
\psfrag{D^b_rb1}{$\Delta^b_{rb,(1)}$}
\psfrag{pa_rD^b_rb2}{$\pa_r \Delta^b_{rb,(2)}$}
\psfrag{pa_rfD^b_rb2}{$\pa_r (f \Delta^b_{rb,(2)})$}
\title{On Cubes of Frobenius Extensions}
\author[Elias, Snyder and Williamson]{Ben Elias, Noah Snyder
  and Geordie Williamson}
\begin{document}

\maketitle

\begin{abstract} Given a hypercube of Frobenius extensions between commutative algebras, we provide a diagrammatic description of some natural transformations between compositions of
induction and restriction functors, in terms of colored transversely-intersecting planar 1-manifolds. The relations arise in the first and third author's work on (singular) Soergel bimodules.
\end{abstract}

%\setcounter{tocdepth}{1}
%\tableofcontents

%%%%%%%%%%%%%%%%%%%%%%%%%%%%
%%%%%%%%%%%%%%%%%%%%%%%%%%%%
%%
\section{Introduction and Results} 
%%
%%%%%%%%%%%%%%%%%%%%%%%%%%%%
%%%%%%%%%%%%%%%%%%%%%%%%%%%%

%===========================
\subsection{Basic Setup}
%===========================

Let $\Bbbk$ be a base field. We will work entirely within the context of commutative $\Bbbk$-algebras.

\begin{defn} A (commutative) \emph{Frobenius extension} is an extension of rings $\iota \co A \into B$, where $B$ is free and finitely generated as an $A$-module, equipped with an
$A$-linear map $\pa = \pa_A^B \co B \to A$, called the \emph{trace}. The trace is required to be \emph{non-degenerate}. That is, we assume that we can equip $B$ with a pair of bases
$\{x_\alpha\}$ and $\{y_\alpha\}$ as an $A$-module, such that $\pa(x_\alpha y_\beta)=\delta_{\alpha \beta}$. These are called \emph{dual bases}. \end{defn}

This data also equips one with a comultiplication map $\Delta^B_A \co B \to B \ot_A B$ sending $1 \mapsto \sum_\alpha x_\alpha \ot
y_\alpha$, an element which is independent of the choice of dual bases. Note that if $A \subset B$ and $B \subset C$ are Frobenius
extensions, then $A \subset C$ is a Frobenius extension as well, with trace $\pa^C_A = \pa^B_A \pa^C_B$.  A typical example to have in mind is $\Bbbk[x^2] \subset \Bbbk[x]$ with the trace $\pa(f) = (f(x)-f(-x))/x$.

A more familiar situation is when $A=\Bbbk$, at which point $B$ is called a Frobenius algebra. Commutative Frobenius algebras are in bijection with 2-dimensional TQFTs. Frobenius
extensions (or Frobenius objects in any category) are no less ubiquitous. An extension $A \subset B$ of commutative algebras is Frobenius if and only if the functors $\Ind^B_A$ and
$\Res^B_A$ are biadjoint. There are numerous standard
examples: \begin{itemize} \item The inclusion $\C[H] \subset \C[G]$ of
  group algebras for an inclusion $H \subset G$ of finite
    abelian (if we want to keep the commutative assumption)
groups. \item The inclusion of symmetric polynomials in all polynomials. \item Various examples constructed using convolution functors in geometry.\end{itemize} For more background
information see \cite{Kad}.

Rather than just a single Frobenius extension, we will be studying several compatible Frobenius extensions.  For example, we might have a square of extensions $A \subset B$, $B \subset C$, $A \subset D$, and $D \subset C$.  More generally, instead of a square of inclusions we might have a larger hypercube.
%The situation we are interested in will be as follows.

\begin{defn} A \emph{hypercube of Frobenius extensions} or a \emph{Frobenius hypercube} will be the following datum. \begin{itemize} \item
A finite set $\Gamma$. We also use $\Gamma$ to designate the entire datum. We consider the hypercube with vertices labelled by subsets of
$\Gamma$. An edge in this hypercube corresponds to $I \setminus \gamma \subset I$ for some $\gamma \in I \subset \Gamma$, and parallel
edges correspond to the same $\gamma$. \item A (contravariant) assignment of rings $R^I$ to vertices in the cube, so that $I \subset J
\implies R^J \subset R^I$. \item For each edge, a trace map $\pa^{I \setminus \gamma}_I \co R^{I \setminus \gamma} \to R^I$ making $R^{I
\setminus \gamma}$ a Frobenius extension of $R^I$. \end{itemize} We call this hypercube \emph{compatible} if, for every square $I \subset
J,J' \subset K$ (so that $I = K \setminus \{\gamma, \gamma'\}$) we have $\pa^J_K \pa^I_J = \pa^{J'}_K \pa^I_{J'}$. In this case,
there is a well-defined map $\pa^I_J \co R^I \to R^J$ for every $I \subset J$, which endows the extension $R^J \subset R^I$ with the
structure of a Frobenius extension. We assume henceforth that every hypercube of Frobenius extensions is compatible. \end{defn}

We sometimes drop the notation for the empty set, so that $\pa_J$ denotes $\pa^{\emptyset}_J$.

\begin{defn} Let $\Bim$ denote the 2-category whose objects are algebras, and where $\Hom_{\Bim}(A,B)$ is the category of finitely
generated $(B,A)$-bimodules. Horizontal composition is given by tensor product. Given a $(B,A)$-bimodule $M$, we sometimes write it as $_B M_A$ to
emphasize which algebras act on it. For a Frobenius hypercube $\Gamma$, we will denote by $\mC(\Gamma)$ the full sub-2-category of $\Bim$
whose objects are the rings of each vertex, and whose 1-morphisms are generated monoidally by the induction bimodule $_B B_A$ and the
restriction bimodule $_A B_B$ for each edge $A \subset B$. \end{defn}

Note that $\mC(\Gamma)$ does not consist of all bimodules isomorphic to compositions of induction and restriction; it consists only of the compositions themselves. Therefore the
objects of $\mC(\Gamma)$ can be encoded combinatorially as paths through the hypercube, and the monoidal structure is concatenation of paths. This monoidal structure is strict, in the
sense that it is associative up to equality of 1-morphisms, not just associative up to isomorphism. The diagrammatic language for describing 2-categories which we use in this paper, and
which is common in the modern literature, is designed to work only with strict 2-categories, which explains why we do not include bimodules isomorphic to compositions over paths.  In particular, for $I \subset J$ which is not an edge, the bimodule $\Ind^I_J$ is not an
object in $\mC(\Gamma)$, even though this bimodule is isomorphic to the composition of inductions along any path from $I$ up to $J$. The influence of the bimodule $\Ind^I_J$ can still be felt in $\mC(\Gamma)$, evident in the existence of natural isomorphisms between the inductions for different paths from $I$ to $J$.

\begin{example} \label{generalSoergelEx} Here is the example which motivated the authors of this paper. Let $\Gamma$ be the vertices of a Dynkin diagram, let $W$ be its Weyl group, and for
any subset $I \subset \Gamma$ let $W_I$ denote the corresponding
parabolic subgroup. Let $R=\C[\mf{h}]$ denote the polynomial ring of
regular functions on the reflection representation of $W$, and let $R^I$
denote the subring invariant under $W_I$. This gives a hypercube of Frobenius extensions, and the 2-category $\mC(\Gamma)$ (or rather, its Karoubi envelope) is the category of singular
Soergel bimodules, as defined by the second author in \cite{Will},
elaborating on ideas of Soergel \cite{Soe}. Understanding the
2-morphisms in this 2-category can help solve natural questions in the
geometry of flag varieties and Kazhdan-Lusztig theory. \end{example}

\begin{example} \label{SoergelEx} The most familiar version of the Soergel cube is where $R=\C[x_1,\ldots,x_n]$, equipped with the natural action of $W=S_n$. This example plays a
key role in the categorification of quantum $\mathfrak{sl}_2$ given by Khovanov and Lauda \cite{KL}. \end{example}

It is well-known that biadjoint functors and the natural transformations between them (such as the 2-morphisms in $\mC(A \subset B)$ for a Frobenius extension) can be described
using diagrams in the planar strip $\R \times [0,1]$. The goal of this paper is to provide a framework whereby 2-morphisms in $\mC(\Gamma)$ for a Frobenius hypercube can be
described by collections of colored oriented 1-manifolds with boundary, and to give some standard relations which hold under some reasonable restrictions.

%===========================
\subsection{Diagrammatics}
%===========================

We assume the reader is familiar with diagrammatics for 2-categories with biadjoints, and the definition of cyclicity for a 2-morphism. An introduction to the topic can be found in
chapter 4 of \cite{Lau}. If $F \co \mc{A} \to \mc{B}$ and $G \co \mc{B} \to \mc{A}$ are biadjoint functors, then one might draw a particular natural transformation as in Figure
\ref{biadjointexample}. Cups and caps correspond to various units and counits of adjunction. Note that one can deduce the labeling of regions from the orientation, or vice versa, so
that some information is redundant. Not every oriented 1-manifold gives rise to a consistent labeling of regions, and only the consistent ones give rise to natural
transformations.

\begin{figure} \caption{A natural transformation from $GFGF$ to $GF$} \label{biadjointexample} $\ig{1.1}{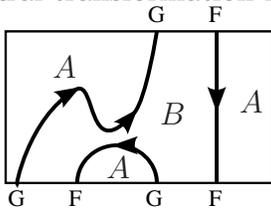}$ \end{figure}

Let $A \subset B$ be a Frobenius extension, and consider diagrams for $\mC(A \subset B)$. We let an upward-oriented line denote the bimodule $_B B_A$
which corresponds to the functor of induction, and a downward-oriented line denote the restriction bimodule $_A B_B$. Technically
speaking, these 1-morphisms are denoted by oriented points, and their identity 2-morphisms by oriented lines, but we shall abuse notation
like this henceforth. A consistent oriented 1-manifold in the planar strip (up to boundary-preserving isotopy) will unambiguously denote a bimodule morphism.
The 4 possible oriented cups and caps correspond to: inclusion $\iota \co A \into B$, trace $\pa \co B \to A$, multiplication $m \co B
\ot_A B \to B$, and comultiplication $\Delta \co B \to B \ot_A B$.

There are additional bimodule morphisms in $\mC(A \subset B)$ which arise from the action of the rings on themselves. Since each ring
$R=A,B$ is commutative, multiplication by $f \in R$ is an $R$-bimodule endomorphism of $R$. We depict this endomorphism as a \emph{box}
containing $f$, located in a region labelled $R$. For an example, see \eqref{polymult} and following. Because the word ``element" is overused, we refer to elements of any ring as
\emph{polynomials}, even though we do not assume that the rings in question are polynomial rings. We sometimes refer to boxes and
polynomials interchangeably.

Now let $\Gamma$ be a hypercube of Frobenius extensions, and consider diagrams for $\mC(\Gamma)$. Regions should be labelled by subsets $I \subset
\Gamma$. The generating 1-morphisms will be $\Ind^{I \setminus \gamma}_I$ and $\Res^{I \setminus \gamma}_I$ for each edge. In addition to
labeling regions, we add the redundant data of placing orientations and colors on each generating 1-morphism. We orient induction and restriction as
above, and color each line based on the element $\gamma$ which is added or subtracted, so that parallel edges have the same color.

Suppose that $A \subset B \subset C$ are Frobenius extensions, so that $\Ind^C_A \cong \Ind^C_B \Ind^B_A$. In other words, there is a natural isomorphism $C \ot_B B \ot_A A \to C$
as $(C,A)$-bimodules, which sends $f \ot g \ot h = fgh \ot 1 \ot 1 \mapsto fgh$. In section 3.2 of \cite{Kh}, a diagrammatic calculus is developed for a 2-category including the
bimodule $\Ind^C_A$ as well as $\Ind^C_B$ and $\Ind^B_A$, which would depict the natural isomorphism above with a trivalent vertex. \igc{1.3}{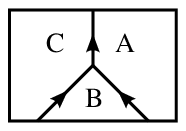} Now suppose that $A \subset
B \subset C$ corresponds to $I \subset I \cup \gamma \subset I \cup \{\gamma, \gamma'\}$ in $\Gamma$. As discussed in the previous section, $\Ind^C_B$ and $\Ind^B_A$ are
objects in $\mC(\Gamma)$, but $\Ind^C_A$ is not, so that there is no use for such a trivalent vertex. However, if $B'$ corresponds to $I \cup \gamma'$ then one has a Frobenius
square $A \subset B,B' \subset C$, and an isomorphism $\phi \colon \Ind^C_B \Ind^B_A \to \Ind^C_{B'} \Ind^{B'}_A$ factoring through $\Ind^C_A$, which sends $f \ot 1 \ot 1 \in C
\ot_B B \ot_A A$ to $f \ot 1 \ot 1 \in C \ot_{B'} B' \ot_A A$. We call this map the \emph{induction isomorphism}, and note that its inverse has the same form.

With our drawing convention above, this map $\phi$ should be drawn as a crossing of two differently-colored 1-manifolds. \igc{1.3}{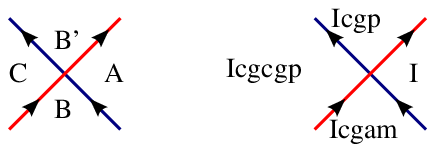} We draw the equivalent
isomorphism $1 \ot 1 \ot f \mapsto 1 \ot 1 \ot f$ for $\Res^C_A$ (the \emph{restriction isomorphism}) as a downward-oriented crossing.

Note that the upward-oriented red strand separating $A$ from $B'$, and the one separating $B$ from $C$, represent two entirely different bimodules; the interpretation of an
upward-oriented red strand depends on the ambient region labels. Though this may cause some initial confusion, it results in a diagrammatic convention whose utility outweighs this minor
hurdle.

The following claim guarantees that an \emph{isotopy class} of diagram
will unambiguously represent a bimodule morphism.  The proof is easy, and appears in Section \ref{further}.

\begin{claim} \label{itscyclic} For a square of Frobenius algebras, the induction isomorphism is cyclic, and rotating it by 180 degrees
yields the restriction isomorphism. \end{claim}

Note that it is impossible for two 1-manifolds of the same color to cross, since that would result in an inconsistent labeling of regions. Given
caps and cups of each color, as well as crossings between different colors, we may produce collections of oriented colored 1-manifolds in
the planar strip, such that the intersection between manifolds of a different color is always transverse. Not every collection of
1-manifolds will be considered, but only the diagrams which result in consistent labelings of regions, which we call \emph{consistent
diagrams}.

Finally, multiplication by a polynomial $f \in R^I$ is an endomorphism of $R^I$ as an $R^I$-bimodule, which we depict as a box. There is now a bimodule morphism interpretation of
any linear combination of consistent diagrams with boxes in various regions, where a box in a region labeled $I$ is itself labeled by a polynomial in $R^I$.

%===========================
\subsection{Relations}\label{relations}
%===========================

Here are some relations among bimodule maps which hold for any Frobenius extension or hypercube. All proofs are found in Section \ref{further}.

Throughout we apply the following convention: Given a
  Frobenius square $\Gamma$ and two subsets $I \subset J$ the set of
  all subsets lying between $I$ and $J$ indexes a smaller Frobenius
  hypercube. Any relations which hold in this smaller hypercube are
  also valid in the larger one. Also, converting a relation in a
  smaller hypercube to that of a larger one simply involves adding the
  subset $I$ to all labels in the relation. We will state all
  relations in their ``minimal'' form (with $|\Gamma| = 1, 2$ or $3$)
  but in proofs we allow ourselves the flexibility of adding any
  indices to relations we have already established.
  
The relations below will use Sweedler notation for coproducts. Suppose that $I \subset J$. The existence of $\Delta^I_{J(1)}$ and
$\Delta^I_{J(2)}$ inside a box in a diagram implies that we take the sum over $\alpha$ (where $\alpha$ indexes dual bases $\{x_\alpha\}$
and $\{y_\alpha\}$ of $R^I$ over $R^J$) of the diagram with $x_\alpha$ replacing $\Delta^I_{J(1)}$ and $y_\alpha$ replacing
$\Delta^I_{J(2)}$. We also use $\mu^I_J$ to represent the product
$\Delta^I_{J(1)}\Delta^I_{J(2)} \in R^I$. We write $\mu_J$ instead of $\mu_J^\emptyset$.

First we have the \emph{isotopy relations}, which are necessary for isotopic diagrams to represent the same map. \begin{equation} \label{isotopy}
\ig{1}{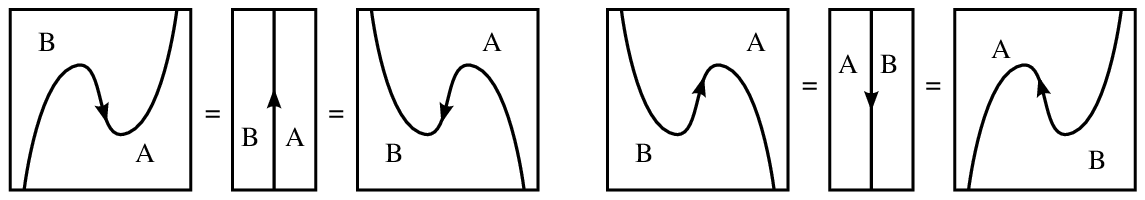} \end{equation} \begin{equation} \label{crosscyclic} \ig{1}{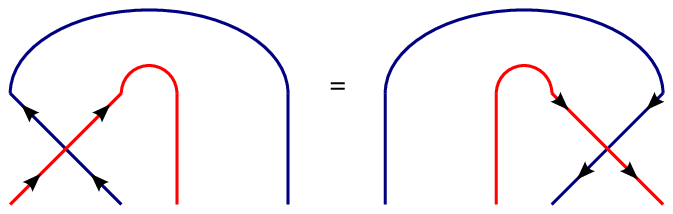} \end{equation} \begin{equation}
\label{crosscyclic2} \ig{1}{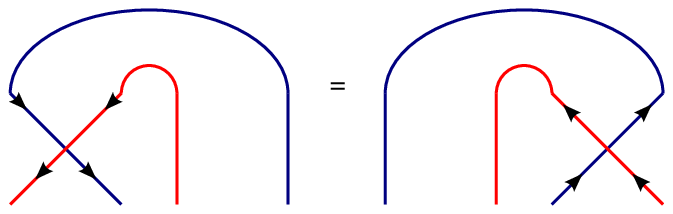} \end{equation}

All future relations hold when rotated or with the colors switched, but not generally with the orientations reversed. Now we have the
relations which hold for any Frobenius extension $\iota \co A \subset B$.

\begin{equation} \label{polymult} \ig{1}{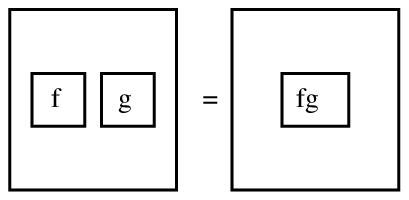} \end{equation}
\begin{equation} \label{polyslide} \ig{1}{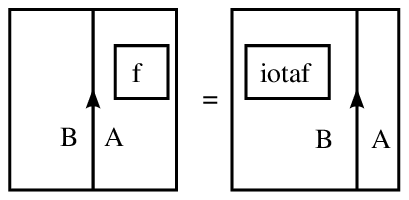} \end{equation}
\begin{equation} \label{cccirc} \ig{1}{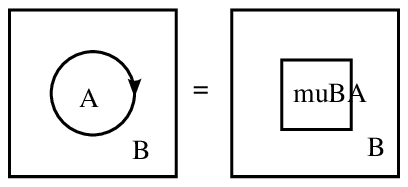} \end{equation}
\begin{equation} \label{ccwcirc} \ig{1}{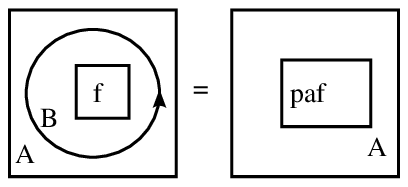} \end{equation}
\begin{equation} \label{Bsplitting} \ig{1}{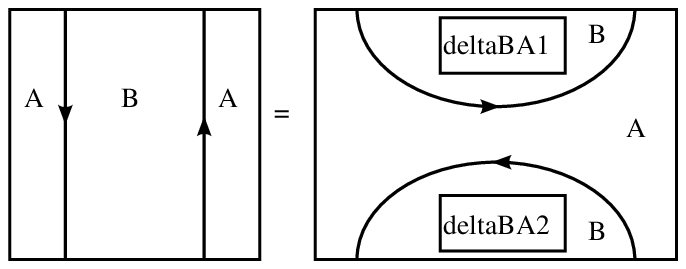} \end{equation}

Now consider a Frobenius square. We call these the \emph{Reidemeister II relations}. We assume that $\Gamma=\{r,b\}$ where the colors red
and blue are assigned to $r$ and $b$.

\begin{equation} \label{R2oriented} \ig{1}{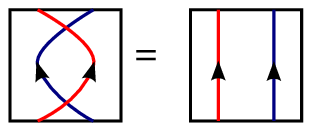} \end{equation}
\begin{equation} \label{R2unoriented1} \ig{1}{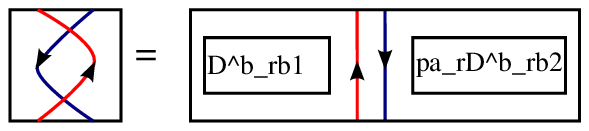} \end{equation}

In fact, there is a more general version of this relation. \begin{equation} \label{R2unoriented1var} \ig{1}{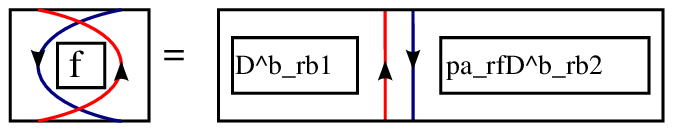} \end{equation} It is implied by (\ref{R2unoriented1})
so long as $R$ is generated by $R^r$ and $R^b$. There are a few
stronger assumptions one might wish to make, which hold in example
\ref{SoergelEx}.

\begin{defn} We say that a Frobenius square satisfies condition $\star$ if we may choose dual bases $\{x_\alpha\}$ and $\{y_\alpha\}$ of $R$ over $R^b$ such that $x_\alpha \in R^r$. We say
that a Frobenius hypercube satisfies $\star$ if every square inside it does. \end{defn}

Condition $\star$ implies that $R$ is generated by $R^r$ and $R^b$, but a priori is stronger.  They are equivalent if all the $R^I$ are fields, or if the $R^I$ are positively graded algebras over a field and the degree $0$ part is just the scalars.   This second case includes the Soergel bimodule examples.

\begin{defn} We say that our Frobenius cube has \emph{no $\mu$-zero
    divisors} if $\mu_\Gamma \in R$ is not a zero divisor (in particular
  $\mu_\Gamma \ne 0$). 
\end{defn}

Assume $\Gamma$ has no $\mu$-zero divisors. Using the identity $\mu_J =
\mu_J^I\mu_I$ we see first that $\mu_I$ is not a zero divisor for any
$I \subset \Gamma$, and then that that $\mu_J^I \in R^I$
are not zero divisors either, for any $I \subset J$. This explains
the terminology.

\begin{remark}
If a Frobenius cube has no $\mu$-zero divisors then each extension
$R^J \subset R^I$ splits when we invert $\mu_I^J$. Hence, after
localizing one can regard the 
cube as being ``semi-simple''. Here it can be easier to check
relations. Furthermore, because all bimodules involved inject into
their localizations, any relations which hold in the localized Frobenius
cube also hold in $\Gamma$.
\end{remark}

\begin{defn} We say that a Frobenius cube satisfies condition R3, if for every triple of colors $r$, $g$, $b$ we have that the ${\mu_{gr} \mu_{rb} \mu_{gb}} \mid {\mu_{grb} \mu_g \mu_r \mu_b}$.
\end{defn}

\begin{remark}
Again, after localization every Frobenius cube satisfies condition $R3$.  In fact, it's easy to see that $\frac{\mu_{grb} \mu_g \mu_r \mu_b}{\mu_{gr} \mu_{rb} \mu_{gb}}$ makes sense if you localize by any of $\mu_{gr}$, $\mu_{rb}$, or $\mu_{gb}$.  Furthermore, if $R$ is a UFD and $\mu_{gr}$, $\mu_{rb}$, and $\mu_{gb}$ are relatively prime then $R3$ automatically holds.
\end{remark}

We now turn to the consequences of these assumptions. When $R$ is generated by $R^r$ and $R^b$, we also have the following relation. \begin{equation} \label{R2unoriented2} \ig{1}{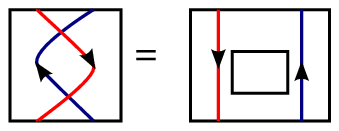} \end{equation} The polynomial in the box is
$\pa_r(\Delta^b_{rb(1)})\Delta^b_{rb(2)}$. When the square satisfies
$\star$  and has no $\mu$-zero divisors, then there is a nicer expression for $\pa_r(\Delta^b_{rb(1)})\Delta^b_{rb(2)}$, which is
$\frac{\mu_{rb}}{\mu_r \mu_b}$. Consequently, $\frac{\mu_{rb}}{\mu_r
  \mu_b}$ is a genuine polynomial, not just an element of the
localization.

Now consider a Frobenius cube. We call these the \emph{Reidemeister III relations}. We assume that $\Gamma=\{r,b,g\}$ where green denotes $g$.

\begin{equation} \label{R3oriented} \ig{1}{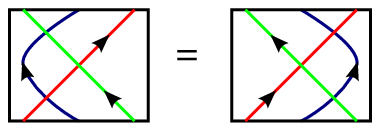} \end{equation}

Our final relation holds if the square satisfies $\star$, has no $\mu$-zero divisors, and satisfies condition R3: it
is \begin{equation} \label{R3unoriented}
  \ig{1}{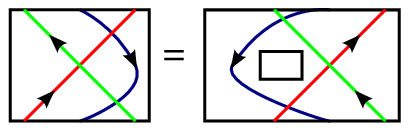} \end{equation} The polynomial in the
box is $\frac{\mu_{grb} \mu_g 
\mu_r \mu_b}{\mu_{gr} \mu_{rb} \mu_{gb}}$, which is a genuine polynomial by assumption R3.

\begin{defn} Suppose that $\Gamma$ satisfies $\star$. We denote by $\mD(\Gamma)$ the diagrammatic 2-category whose objects are subsets $I \subset \Gamma$, whose 1-morphisms are
generated by up and down arrows for each edge of the hypercube, and whose 2-morphisms are (linear combinations of) consistent diagrams of colored transversely-intersecting oriented
planar 1-manifolds with boxes, modulo the relations above. Let $\mF$ be the obvious 2-functor $\mD(\Gamma) \to \mC(\Gamma)$, sending cups caps and crossings to the appropriate
bimodule maps. \end{defn}

 We do not claim that $\mF$ is full or faithful. In any reasonable example there will be additional relations which do not hold in the general case, so that $\mF$ will not be
faithful. For the Soergel bimodule example $\mF$ is full, though we do not know if $\mF$ will be full in general. We also do not assert that this is a complete list of relations for
a generic Frobenius hypercube, or that every generic relation can be expressed with diagrams using at most 3 colors.

\begin{remark} \label{simplify2} It is a consequence of the relations that any diagram in $\mD(\Gamma)$ with only 2 colors can be ``simplified," i.e. expressed as a linear
combination of diagrams without any closed 1-manifold components. Any 2-color diagram without boundary will reduce to a box. The simplification procedure is to pull apart strands
using the Reidemeister II relations, until we separate a closed component from the rest of the diagram, and reduce it to a box using one-color relations. 

% Even though we may always
% pull strands apart, we may \emph{not} cross them back over each other using relation \eqref{R2unoriented1} or \eqref{R2unoriented2} unless the appropriate boxes are present.

Diagrams with 3 colors can not be simplified: we are only allowed to apply the Reidemeister III move (\ref{R3unoriented}) in one direction, unless the appropriate polynomials
are present. This deficiency can not be remedied: for Example \ref{SoergelEx}, there are bimodule morphisms which can not be expressed by diagrams without closed 1-manifolds.
\end{remark}

%\begin{remark} We quickly comment on the form of the polynomials in relations (\ref{R2unoriented2}) and (\ref{R3unoriented}). It is easy to show that $\mu^I_K = \mu^I_J \mu^J_K \in R^I$ whenever $I \subset J \subset K$. One can show that there exists one polynomial $\Pi_I$ for each subset $I \subset \Gamma$, with $\Pi_\emptyset=1$, such that $\mu_I$ is the product of $\Pi_J$ over all $J \subset I$. Then $\frac{\mu_{rb}}{\mu_r \mu_b} = \Pi_{rb}$, and $\frac{\mu_{grb} \mu_g \mu_b \mu_r}{\mu_{rb} \mu_{rg} \mu_{gb}} = \Pi_{grb}$.

%For example \ref{SoergelEx}, these polynomials are easy to compute. When $I = \{i,i+1,\ldots,j\}$ is connected, then $\Pi_I= x_i- x_{j+1}$ is the highest root of $W_I$. When $I$ is not connected, $\Pi_I=1$. {\color{blue} Note, however, that the polynomials appearing in relation \eqref{R2unoriented2} (resp. \eqref{R3unoriented}) depend on which 2-dimensional (resp. 3-dimensional) cube is chosen within the $n$-dimensional hypercube. If $K$ is the label of the innermost region of the LHS of \eqref{R3unoriented}, and $\{i_1, i_2, i_3\}$ are the colors of the three strands, then the polynomial which appears on the RHS is the product of all the roots in $W_K$ which are not roots of $W_{K \setminus i_m}$ for $m=1,2,3$. }\end{remark}

%{\color{red}GW: I don't see why we should get genuine polynomials unless $R$ is a UFD. I rewrote this remark as follows:}

\begin{remark} We quickly comment on the form of the polynomials in
  relations (\ref{R2unoriented2}) and (\ref{R3unoriented}) under the
  assumption that $\Gamma$ has no $\mu$-zero divisors. It is easy
  to show that $\mu^I_K = \mu^I_J 
\mu^J_K \in R^I$ whenever $I \subset J \subset K$. One can use this to
show the existence of an element $\Pi_I \in R[\mu_\Gamma^{-1}]$ for
each subset $I \subset \Gamma$, with $\Pi_\emptyset=1$, such that 
$\mu_I$ is the product of $\Pi_J$ over all $J \subset I$. For example if $\Gamma = \{ r, b \}$ then using that
\[
\mu_{rb} = \mu_{rb}^b \mu_b = \mu_{rb}^r\mu_r
\]
we have $\Pi_{rb} = \mu_{rb}^b / \mu_r = \mu_{rb}^r / \mu_b =
\mu_{rb}/\mu_r\mu_b$. Similarly  $\frac{\mu_{grb} \mu_g
  \mu_b \mu_r}{\mu_{rb} \mu_{rg} \mu_{gb}} = 
\Pi_{grb}$ etc.

In the setting of \ref{SoergelEx}, these polynomials are easy to
compute. When $I = \{i,i+1,\ldots,j\}$ is an interval, then $\Pi_I= x_i
- x_{j+1}$ is the highest root of $W_I$. When $I$ is
not an interval, $\Pi_I=1$. Similar statements hold for any finite
Weyl group, in the setting of \ref{generalSoergelEx}.\end{remark}

\begin{remark} \label{history} A finite index von Neumann subfactor $N \subset M$ is an example of a Frobenius extension.  Chains and squares of extensions have been studied before by people who work in the field of subfactors \cite{BJ,GJ,Lan,SW}. In that context, the $2$-category defined above is called the standard invariant, and similar diagrammatics have been developed for standard invariants of lattices of subfactors. However, there is a major difference between the subfactor world and ours: the rings they use are non-commutative but their center is trivial.  Thus the complicated parts of their theory (non-commuting or non-cocommuting quadrilaterals) do not appear here, while the complicated part our our theory (the behavior of the boxes) does not appear there.  For example,  in the subfactor setting a relation like (\ref{R2unoriented1}) or (\ref{R3unoriented}) must be trivial if it exists, in the sense that the polynomial(s) appearing are equal to $1$.
\end{remark}

%\begin{remark} \label{history} Chains and squares of extensions have been studied before by people who work in the field of subfactors \cite{BJ,GJ,Lan,SW}. Similar diagrammatics
%are developed (there and elsewhere), although usually a version with trivalent vertices. However, there is a major difference between the subfactor world and ours: the ``rings"
%they use are non-commutative and their center is trivial, so there are no boxes. As a consequence, a relation like (\ref{R2unoriented1}) or (\ref{R3unoriented}) must be trivial if
%it exists, in the sense that the polynomial(s) appearing are equal to $1$ (this is the commuting and cocommuting case mentioned in \cite{GJ,SW}). \end{remark}

%%%%%%%%%%%%%%%%%%%%%%%%%%%%
%%%%%%%%%%%%%%%%%%%%%%%%%%%%
%%
\section{Further Details and Proofs}
\label{further} 
%%
%%%%%%%%%%%%%%%%%%%%%%%%%%%%
%%%%%%%%%%%%%%%%%%%%%%%%%%%%

%=====================
\subsection{Frobenius Extensions}
%=====================

Suppose that $A \subset B$ is a Frobenius extension. The following statements are all standard, and hold for all $f \in B$.

\begin{equation} f\Delta = \Delta f \in B \ot_A B\end{equation}
\begin{equation} \Delta_{(1)} \ot \pa(f \Delta_{(2)}) = f \ot 1 \in B \ot_A B \end{equation}
\begin{equation} \Delta_{(1)} \pa(f \Delta_{(2)}) = f \in B \end{equation}

These three equations are sufficient to prove (\ref{isotopy}) and (\ref{Bsplitting}). Relations (\ref{polymult}) and (\ref{polyslide}) are
obvious properties of polynomials. Relations (\ref{cccirc}) and (\ref{ccwcirc}) can be checked on the element $1$, and follow immediately.
The relation (\ref{Bsplitting}) implies the splitting of $B$ into a free $A$ module by decomposing the identity element into a sum of
orthogonal idempotents.

\begin{claim} \label{1colorequiv} The 2-functor $\mF \co \mD(A \subset B) \to \mC(A \subset B)$ is an equivalence of categories.
\end{claim}

\begin{proof} Let $M$ and $N$ be any two 1-morphisms between the same two objects in $\mD=\mD(A \subset B)$. The relations of $\mD$ imply biadjointness of $_A B_B$ and $_B B_A$, and
the isomorphism $_A B_A \cong (_A A_A)^{\oplus n}$ where $n$ is the rank of $B$ over $A$. Using only these two facts, it is a simple exercise to express $\Hom_{\mD}(M,N)$ as a
direct sum of copies of $\End(_A A_A)$ and $\End(_B B_B)$. In $\mC$, the same expression for $\Hom_{\mC}(M,N)$ works, and the functor $\mF$
preserves both the adjunction morphisms and the direct sum decomposition, meaning that this expression for $\Hom(M,N)$ is functorial under $\mF$. Therefore, $\mF$ is fully faithful
if and only if it induces isomorphisms on $\End(_A A_A)$ and $\End(_B B_B)$.

Since there is only one color, it is a simple inductive argument to show that any nested combination of circles and boxes in $\mD$ will reduce to a box. In particular, an endomorphism of
an empty region labelled $A$ reduces to a box labelled by $f \in A$, so that it is generated as an $A$-module by the identity map. This $A$-module maps under $\mF$ to the free rank 1
$A$-module $\End_\mC(_A A_A)$, and therefore this map is an isomorphism. The same holds true for $B$. \end{proof}

%=====================
\subsection{Chains of Frobenius Extensions}\label{frobchain}
%=====================

Suppose that $A \subset B \subset C$ is a chain of Frobenius extensions. We equip $C$ over $A$ with a trace map $\pa^C_A = \pa^B_A
\pa^C_B$. If $\{x_\alpha\}$ and $\{y_\alpha\}$ are dual bases of $B$ over $A$, and $\{p_\beta\}$ and $\{q_\beta\}$ are dual bases of $C$
over $B$, then $\{x_\alpha p_\beta\}$ and $\{y_\alpha q_\beta\}$ are dual bases of $A$ over $C$. This makes it easy to see that
$\Delta^C_A = \Delta^B_A \Delta^C_B$, where $\Delta^C_B \colon C \to C \ot_B C \cong C \ot_B B \ot_B C$, and $\Delta^B_A$ is applied to
the middle factor to reach $C \ot_B B \ot_A B \ot_B C \cong C \ot_A C$. In general, we will always identify $C \ot_B B \ot_A A$ with $C
\ot_A A$ using the canonical isomorphism. We have:

\begin{equation} \Delta^C_A = \Delta^B_A \Delta^C_B \end{equation}
\begin{equation} \mu^C_A = \mu^C_B \mu^B_A \end{equation}
\begin{equation} \label{eqn1} \Delta^C_{A(1)} \ot \partial^C_B(\Delta^C_{A(2)}) = \iota^C_B(\Delta^B_{A(1)}) \ot \Delta^B_{A(2)}.
	\end{equation}
\begin{equation} \label{eqn2} f \iota^C_B(\Delta^B_{A(1)}) \ot \Delta^B_{A(2)} = \Delta^C_{A(1)} \ot \partial^C_B(f \Delta^C_{A(2)}) \textrm{ for
	any } f \in C. \end{equation}
\begin{equation} \Delta^C_{A(1)} \partial^C_B(\Delta^C_{A(2)}) = \mu^B_A \end{equation}	
	
Similar statements hold when applying the operator $\partial$ to the left side instead of the right. We shall always assume this
right-left symmetry.

The interesting equations (\ref{eqn1}) and (\ref{eqn2}) follow from unraveling this equality. \igc{1}{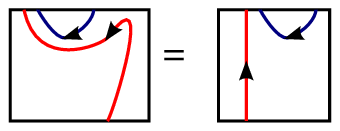}

%=====================
\subsection{Squares of Frobenius Extensions}\label{frobsquare}
%=====================

Suppose that $\Gamma = \{r,b\}$. \igc{1}{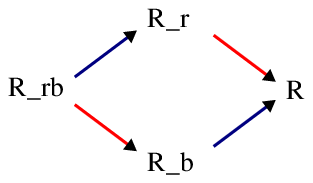} Recall that we label trace maps, comultiplications and the like by subsets of $\Gamma$. For instance, $\partial_{rb}^r \colon R^r \to R^{rb}$, or $\Delta_b \colon R \to R \ot_b R$, which is short for $R \ot_{R^b}
R$. Assume that this square is compatible, so that $\partial^r_{rb} \partial_r = \partial^b_{rb} \partial_b = \partial_{rb}$. This implies
the dual statement about comultiplication. We shall let inclusion maps $\iota$ be implied, so that $\partial_b(f)$ for $f \in R^r$
denotes $\partial_b(\iota_r(f))$.

\begin{claim} Relations (\ref{crosscyclic}) and (\ref{crosscyclic2}) hold, so that an isotopy class of diagram unambiguously represents a
morphism. \end{claim}

\begin{proof} For (\ref{crosscyclic}) these diagrams represent maps $R \to R^{rb}$ as $R^{rb}$-bimodules. The LHS sends $f \in R$ to
$\pa^r_{rb} \pa_r f$ and the RHS to $\pa^b_{rb} \pa_b f$, which we have assumed are equal. For (\ref{crosscyclic2}) these diagrams
represent maps $R \ot_{rb} R \to R$, and both sides are simply multiplication. \end{proof}

The oriented Reidemeister II move (\ref{R2oriented}) is obvious. To examine the non-oriented Reidemeister II moves, we find formulae for
sideways crossings. The left-pointing crossing gives a map from $R \to R^r \ot_{rb} R^b$, for which the image of $f$ is clearly the
element in the next equality.

\igc{1}{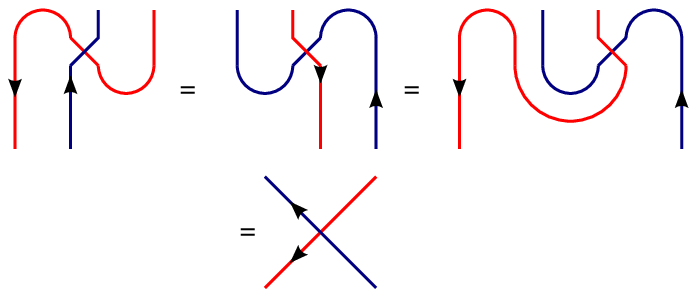}

\begin{equation} \label{sidewayseq}
\Delta^r_{rb(1)} \ot \partial_b(f\Delta^r_{rb(2)}) = \partial_r(f \Delta^b_{rb(1)}) \ot \Delta^b_{rb(2)} =
\partial_r(f \Delta_{rb(1)}) \ot \partial_b(\Delta_{rb(2)}) \in R^r \ot_{rb} R^b. \end{equation}

The right-pointing crossing gives a map $R^r \ot_{rb} R^b \to R$, which can easily be verified to be the multiplication map. The counterclockwise Reidemeister II move (\ref{R2unoriented1})
and its analog (\ref{R2unoriented1var}) are now obvious. Because the maps are $(R^r,R^b)$-linear, one can check these equalities on the image of $1 \ot 1$.

Consider the clockwise Reidemeister II move (\ref{R2unoriented2}), an equality of endomorphisms of $R$ as an $(R^r,R^b)$-bimodule. The RHS is a morphism which is $R$-linear, while the LHS is
not obviously $R$-linear. The image of $1$ under both sides is $\pa_r(\Delta^b_{rb(1)}) \Delta^b_{rb(2)}$, or any of its equivalent presentations above. Therefore the relation will hold if
and only if the LHS is $R$-linear. Suppose that $R$ is generated as an algebra by its subalgebras $R^r$ and $R^b$. Then the $(R^r,R^b)$-bilinearity of the LHS will actually imply
$R$-linearity.

\begin{claim} If the square satisfies $\star$ and has no
    $\mu$-zero divisors then $$\pa_r(\Delta^b_{rb(1)})\Delta^b_{rb(2)} = \frac{\mu_{rb}}{\mu_r \mu_b}.$$ \end{claim}

\begin{proof} Choose dual bases $\{x_\alpha\}$ and $\{y_\alpha\}$ of $R$ over $R^b$ such that $x_\alpha \in R^r$. Also choose dual bases $\{p_\beta\}$ and $\{q_\beta\}$ of $R^b$ over
$R^{rb}$. Then we are examining $\sum_\beta \pa_r(p_\beta) q_\beta$. Multiplying this by $\mu_b$, we obtain $$\sum_{\alpha,\beta} x_\alpha \pa_r(p_\beta) y_\alpha q_\beta =
\sum_{\alpha,\beta} \pa_r(x_\alpha p_\beta) y_\alpha q_\beta \\ = \pa_r(\Delta_{rb(1)}) \Delta_{rb(2)} = \mu^r_{rb} = \frac{\mu_{rb}}{\mu_r}.$$ So dividing by $\mu_b$ again, we get
$\frac{\mu_{rb}}{\mu_r \mu_b}$. \end{proof}

\begin{prop} Suppose that all the Reidemeister II relations hold. Any morphism in $\mD(\Gamma)$ for a Frobenius square is a linear combination of diagrams with no closed components of
either color. \end{prop}

\begin{proof} We show this by induction on subdiagrams, and on the total number of components. The base case is the empty diagram, perhaps with a box. Given a general diagram, let us
assume inductively that all proper closed subdiagrams may be reduced to boxes. Suppose there is a closed red component. If any blue strands intersect it, we choose an ``innermost"
strand, so that there is an instance of the LHS of some Reidemeister II relation without any additional blue strands crossing the red one in the picture. There may be other junk on the
interior of the picture, but the interior is a proper closed subdiagram so we may reduce the interior to a box. Depending on the orientation, we may either slide the box out and apply
(\ref{R2unoriented2}) or (\ref{R2oriented}), or may simply apply (\ref{R2unoriented1var}), so that the blue strand no longer intersects the red component. Repeating this argument, we
may assume that no blue strands intersect the red component. Then we reduce the interior of the component to a box, and use the circle relations (\ref{cccirc}) and (\ref{ccwcirc}) to
eliminate the red component in question. This entire procedure may have added boxes in various regions, but did not otherwise affect the topology of the diagram except by removing
components. By induction, the remaining diagram can be reduced. \end{proof}

\begin{remark} We do not have an analog of Claim \ref{1colorequiv} for the case of two colors. By simplifying diagrams, we know that the endomorphisms of an empty region labelled
$I$ are isomorphic to $R^I$, as desired. However, adjunction and direct sum decompositions are not sufficient to reduce any Hom space to this form. In Example \ref{generalSoergelEx} for a finite rank 2 Coxeter group, an additional relation is required. \end{remark}

\subsection{Cubes of Frobenius Extensions}\label{frobcube}

Suppose that $\Gamma = \{r,b,g\}$, and we have a compatible cube of Frobenius extensions.

For the obvious reasons, any upward-oriented Reidemeister III equality (\ref{R3oriented}) will hold.

Consider diagrams which look like Reidemeister III but with different orientations. Any picture where the interior triangle does
\emph{not} have an oriented boundary will be a rotation of the upward-oriented Reidemeister III move, and thus we are allowed to slide one
line over the crossing. Any picture with an interior triangle which has clockwise or counter-clockwise orientation will not permit such a
slide, requiring a relation like (\ref{R3unoriented}). Our proof assumes condition $\star$, that there are no $\mu$-zero divisors, and condition $R3$.

%{\color{red}  GW: are you sure about this? Perhaps we should say something like: we have not be able to find a proof unless we make some additional assumptions.}

\begin{proof}[Proof of \eqref{R3unoriented}]
We can
  write
\[
\Pi_{grb} = \frac{ \mu_{rgb}\mu_r\mu_g\mu_b}{\mu_{rg}\mu_{rb}\mu_{gb}}
  = pq^{-1}
\]
where
\[
p = \frac{\mu_{grb} \mu_g} {\mu_{rg} \mu_{bg}} =
\frac{\mu_{rgb}^g}{\mu_{rg}^g\mu_{bg}^g}
  \quad \text{and} \quad q = \frac {\mu_{rb}}
{\mu_r \mu_b}.
\]
Then we have
$$
{
\labellist
\small\hair 2pt
 \pinlabel {$pq^{-1}$} [ ] at 45 202
 \pinlabel {$=$} [ ] at 97 210
 \pinlabel {$q^{-1}$} [ ] at 145 203
 \pinlabel {$p$} [ ] at 145 172
 \pinlabel {$=$} [ ] at 189 210
 \pinlabel {$q^{-1}$} [ ] at 237 209
 \pinlabel {$=$} [ ] at 31 64
 \pinlabel {$q^{-1}$} [ ] at 84 91
 \pinlabel {$=$} [ ] at 141 64
\endlabellist
\centering
\ig{1}{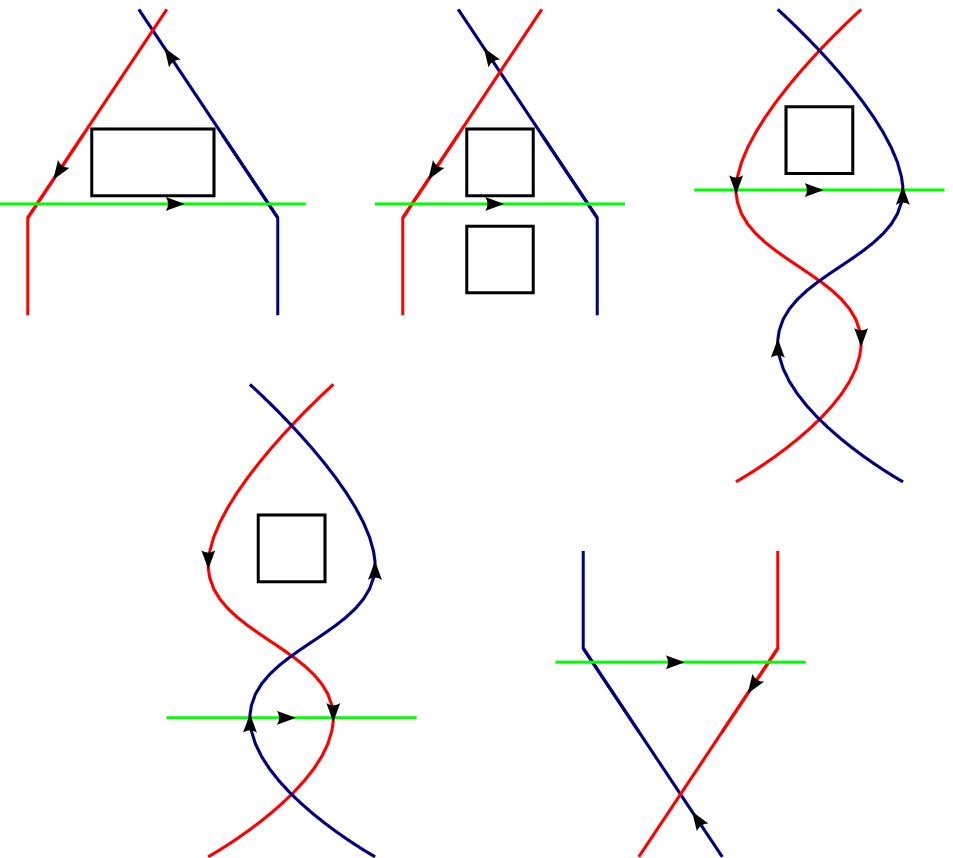}
}
$$
where for the last equality we are claiming
$$
{
\labellist
\small\hair 2pt
 \pinlabel {$q^{-1}$} [ ] at 18 38
 \pinlabel {$=$} [ ] at 61 38
\endlabellist
\centering
\ig{1}{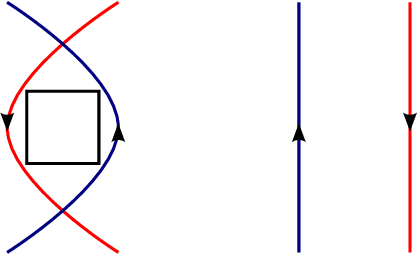}
}
$$
In other words, by \eqref{R2unoriented1var} we have reduced the proof
to the identity:
\begin{equation} \label{finalid}
\Delta_{rb,(1)}^b \otimes \partial_r (q^{-1} \Delta_{rb,(2)}^b) = 1
\otimes 1 \quad \text{in $R^b \otimes_{R^{rb}} R^r$}.
\end{equation}
Note that
\[ q^{-1} = (\mu_{rb}^b)^{-1} \mu_r \]
and by condition $\star$ we can choose dual bases $\{ c_\alpha \}$ and $\{
d_\alpha\}$ for $R^r$ over $R$ such that $c_\alpha \in R^b$. Then we have
\[
q^{-1} = \sum_\alpha (\mu_{rb}^b)^{-1}c_\alpha d_\alpha
\]
and we get identity \eqref{finalid} as follows:
\begin{align*}
\Delta_{rb,(1)}^b \otimes \partial_r (q^{-1} \Delta_{rb,(2)}^b) & =
\sum_\alpha \Delta_{rb,(1)}^b \otimes \partial_r(  (\mu_{rb}^b)^{-1}  c_\alpha d_\alpha
\Delta_{rb,(2)}^b) \\
& \stackrel{(1)}{=} \sum_\alpha (\mu_{rb}^b)^{-1} \Delta_{rb,(1)}^b c_\alpha
\otimes \partial_r ( d_\alpha \Delta_{rb,(2)}^b) \\
& \stackrel{(2)}{=} (\mu_{rb}^b)^{-1} \Delta_{rb,(1)}^b \Delta_{rb,(2)}^b \otimes 1 = 1
\otimes 1.
\end{align*}
For (1) we have used \eqref{sidewayseq}, which implies that 
\[ \Delta_{rb,(1)}^b \otimes \partial_r(  (fg
\Delta_{rb,(2)}^b) = f \Delta_{rb,(1)}^b \otimes \partial_r(  (g
\Delta_{rb,(2)}^b) \quad \text{for $f \in R^b$.}
\]
Finally, (2) follows from \eqref{eqn2}
with $A = B = R^r$ and $C = R$.
\end{proof}


\begin{thebibliography}{10}
	
\bibitem{BJ} D.~Bisch and V.~F.~R.~Jones, Algebras associated to intermediate subfactors, \emph{Invent. Math.} {\bf 128} (1997), 89--157.

\bibitem{GJ} P.~Grossman and V.~F.~R.~Jones, Intermediate subfactors with no extra structure, preprint 2005, arXiv:math/0412423v3.
	
\bibitem{Kad} L.~Kadison, New examples of Frobenius algebras, \emph{AMS University Lecture Series} {\bf 14} (1999).	

\bibitem{Kh} M.~Khovanov, Heisenberg algebra and a graphical calculus, preprint 2010, arXiv:1009.3295v1.

\bibitem{Lan} Z.~Landau, Fuss-Catalan algebras and chains of intermediate subfactors, \emph{Pac. J. of Math.} {\bf 197} (2001), no. 2, 325--367.

\bibitem{KL} M.~Khovanov and A.~Lauda, A diagrammatic approach to categorification of quantum groups I, \emph{Represent. Theory} {\bf 13} (2009), 309--347.

\bibitem{Lau} A.~Lauda, A categorification of quantum sl(2), \emph{Adv. in Math.} {\bf 225} (2010), no. 6 , 3327--3424. arXiv:0803.3652.

\bibitem{SW} T.~Sano and Y.~Watatani, Angles between two subfactors, \emph{J. Operator Thry.} {\bf 32} (1994), 209--241.

\bibitem{Soe} W.~Soergel, The combinatorics of Harish-Chandra bimodules, \emph{J. Reine Angew. Math.} {\bf 429} (1992), 49--74.

\bibitem{Will} G.~Williamson, Singular Soergel Bimodules, \emph{Int. Math. Res. Not.} 2011, No. 20, 4555-4632 (2011).

\end{thebibliography}
\end{document}